\documentclass[11pt]{amsart}
\usepackage[T1]{fontenc}
\numberwithin{equation}{section}

\usepackage{float}
\usepackage{tikz-cd}
\usepackage{amsmath}%
\usepackage{amsfonts}%
\usepackage{amsthm}
\usepackage{amssymb}%
\usepackage{graphicx}
\usepackage{hyperref}

\usepackage{float}
\usepackage{tikz-cd}
\usepackage{amsmath}%
\usepackage{amsfonts}%
\usepackage{amsthm}
\usepackage{amssymb}%
\usepackage{graphicx}
\usepackage{hyperref}
\usepackage{enumerate}
\usepackage{comment}

\usepackage{geometry}
\usepackage{layout}
\usepackage[utf8]{inputenc}
\usepackage{amsmath}
\usepackage{graphicx}
\usepackage{amsmath}
\usepackage{amssymb}
\usepackage{amsfonts}
\usepackage{graphicx}
\usepackage{url}
\usepackage{MnSymbol}
\usepackage{arydshln}
\usepackage{outlines}
\usepackage{tikz-cd}

\usepackage{subfig}
\usepackage{amsthm}

\newcommand{\Fbar}{\overline{\mathbb{F}}}

\newcommand{\Z}{\mathbb{Z}}
\newcommand{\Q}{\mathbb{Q}}
\newcommand{\CC}{\mathbb{C}}
\newcommand{\cA}{\mathcal{A}}
\newcommand{\cM}{\mathcal{M}}
\newcommand{\cS}{\mathcal{S}}
\newcommand{\cJ}{\mathcal{J}}
\newcommand{\F}{\mathbb{F}}
\newcommand{\PGL}{\mathrm{PGL}}
\renewcommand{\O}{\mathcal{O}}
\renewcommand{\epsilon}{\varepsilon}

\renewcommand{\P}{\mathbb{P}}

\renewcommand{\cJ}{\mathcal{J}}
\newcommand{\Jac}{\mathcal{J}}

\newcommand{\dirlim}{\varinjlim}

\renewcommand{\o}{\varnothing}

\renewcommand{\o}{\varnothing}
\newcommand{\Car}{\mathcal{C}}

\newcommand{\rank}{\mathrm{rank}}

\newcommand{\Aut}{\mathrm{Aut}}

\newcommand{\Hom}{\mathrm{Hom}}

\newcommand{\prank}{p\text{-}\rank}
\newcommand{\tworank}{2\text{-}\rank}
\newcommand{\dd}{\mathrm{d}}
\usepackage{multirow}

\newcommand{\<}{\left \langle}
\renewcommand{\>}{\right \rangle}

\usepackage{xcolor}

\newtheorem{thm}{Theorem}[section]
\newtheorem{lem}[thm]{Lemma}
\newtheorem{prop}[thm]{Proposition}
\newtheorem{cor}[thm]{Corollary}

\theoremstyle{definition}

\theoremstyle{remark}
\newtheorem{rem}[thm]{Remark}
\newtheorem{exmp}[thm]{Example}

\begin{document}

\author{Du\v san Dragutinovi\'c}
\address{Mathematical Institute, Utrecht University,
P.O. Box 80010, 3508 TA Utrecht, The Netherlands}
\email{d.dragutinovic@uu.nl}

\title[Supersingular curves of genus four in characteristic two]{Supersingular curves of genus four\\ in characteristic two}

\maketitle

\begin{abstract}
We describe the intersection of the Torelli locus $j(\cM_4^{ct}) = \cJ_4$ with Newton and Ekedahl-Oort strata related to the supersingular locus in characteristic two. We show that the locus of supersingular Jacobians $\cS_4\cap\cJ_4$ in characteristic two is pure of dimension three. One way to obtain that result uses an analysis of the data of smooth genus four curves and principally polarized abelian fourfolds defined over $\mathbb{F}_2$, and another involves studying the Ekedahl-Oort types and using the indecomposability of some $\mathrm{BT}_1$ group schemes.
\end{abstract}

\section{Introduction}

Suppose $k$ is a field of characteristic $p>0$ and $g\geq 2$. We would like to describe the loci of curves $C$ defined by invariants of the $p$-torsion of its Jacobian, for $C$ either a smooth (irreducible, projective, algebraic) genus-$g$ curve over $k$ or a genus-$g$ stable curve over $k$ of compact type (i.e., a stable curve whose dual graph is a tree).

Let $\cM_g$ be the moduli space of smooth genus-$g$ curves, $\cM^{ct}_g$ the moduli space of stable genus-$g$ curves of compact type, and $\cA_g$ the moduli space of principally polarized $g$-dimensional abelian varieties. To a genus-$g$ stable curve of compact type $C$, we can attach its Jacobian variety $\Jac_C$, and this induces the Torelli morphism $$j: \cM^{ct}_g \to \cA_g.$$ We call $j(\cM_g^{ct}) = \cJ_g$ the Torelli locus, and $j(\cM_g) = \cJ_g^0 \subset \cJ_g$ the open Torelli locus. It is well-known that $\dim \cM_g = \dim \cM^{ct}_g = \dim \cJ_g^0 = \dim \cJ_g =  3 g - 3$, and $\dim \cA_g = \frac{g(g + 1)}{2}$.

We say that an abelian variety $A$ over $k$ of dimension $g>0$ is \textit{supersingular} if there is an isogeny $$A \sim E^g,$$ where $E$ is a supersingular elliptic curve over $\bar{k}$, that is $E[p](\bar{k}) = \{O\}$. 
We say that a curve $C$ of compact type is supersingular if its Jacobian is supersingular.  With $\cS_g \subset \cA_g$, we denote the locus of supersingular $g$-dimensional abelian varieties. A lot of things are known about $\cS_g$; e.g.,  \cite{lioort}, Theorem 4.9 gives $$\dim \cS_g = \left \lfloor \frac{g^2}{4} \right \rfloor, $$ for any $g\geq 1$. Furthermore, using the theory of Dieudonn\'e modules and the Hecke correspondence, some stratifications of $\cA_g$ over $k$ that we mention below are described very well; see for example \cite{normanoort}, \cite{lioort}, and \cite{chaioort}.

We are mainly interested in the locus of supersingular curves, and hence in $\cS_g\cap \cJ_g$, the locus of supersingular $g$-dimensional Jacobians. For $g = 2$ or $g = 3$, $\cJ_g^0$ is dense in $\cA_g$, so we get information about the supersingular curves directly. However, when $g\geq 4$, we know very little about the supersingular curves in general. For example, even for $g = 4$, it is only recently shown in \cite{kudoharashitasenda} that  for an arbitrary prime $p>0$ there is a supersingular smooth curve of genus $4$. The corresponding question for  arbitrary $p>0$ and $g\geq 5$ is still open. Furthermore, the knowledge about the structure of the locus of supersingular curves and some related loci is very limited. See \cite{pries_current_results}, where some of the open questions are mentioned. The biggest obstacle is that we cannot directly use the tools developed for working in $\cA_g$. 

For principally polarized $g$-dimensional abelian varieties $A$ over $k$, there are several invariants that introduce the stratifications of $\cA_g$: \textit{$p$-rank}, \textit{Newton polygon}, \textit{Ekedahl-Oort type}, and \textit{$a$-number}. The \textit{$\prank$} of $A$ is defined as the number $f$, with $0\leq f\leq g$, so that $$\#A[p](\bar{k}) = p^f.$$ Let $A[p^{\infty}] = \dirlim A[p^n]$ be the $p$-divisible group of $A$. By the Dieudonn\'e-Manin classification \cite{manin}, there are certain $p$-divisible groups $G_{c, d}$ for $c, d\geq 0$ relatively prime integers, so that $A[p^{\infty}]$ is up to isogeny (of $p$-divisible groups) equal to $$\oplus_{\lambda = \frac{d}{c + d}}G_{c, d},$$ for a unique choice of \textit{slopes} $\lambda$. The Newton polygon of $A$ is defined as the collection of those $\lambda$ counted with multiplicities. It holds that $$A \text{ is supersingular if and only if } A[p^{\infty}] \sim (G_{1, 1})^g.$$ 
The $\prank$ and the Newton polygon are isogeny invariants. The number of slopes $\lambda = 0$ equals the $\prank$ of $A$, and in particular, the supersingular locus $\cS_g$ is contained in the $\prank$ zero locus $V_0$. The isomorphism class of $p$-torsion group schemes $A[p]$ is determined by the Ekedahl-Oort type of $A$, which we present with the Young type $$\mu = [\mu_1, \mu_2, \ldots, \mu_n], $$ with $g\geq \mu_1 > \ldots > \mu_n > 0$; equivalently, the Ekedahl-Oort type is also determined by the final type $\nu = \{\nu(1), \nu(2), \ldots, \nu(g)\}$, connected with $\mu$ via $\mu_j = \#\{i: 1\leq j \leq g, \nu(i) + j \leq g\}$. The $a$-number of $A$ is $$a(A) = \dim_k\Hom(\alpha_p, A),$$ where $\alpha_p$ is the group scheme defined as the kernel of Frobenius on the additive group $\mathbb{G}_a$. If $A$ has the Ekedahl-Oort type $[\mu_1, \ldots, \mu_n]$, then $\prank(A) = g - \mu_1$ and $a(A) = n$. In general, we have $1\leq \prank(A) +  a(A) \leq g$. For a smooth curve $C$, we define these invariants in terms of its Jacobian $\Jac_C$ - e.g., $\prank(C) = \prank(\Jac_C)$. See \cite{oort_mixed_char} or \cite{pries_current_results} for more details. 
\\

As one of the first steps, we describe what happens when $g = 4$ and $p = 2$ and consider the previously introduced moduli spaces over $\Fbar_2$, that is, $\cM_4 = \cM_4 \otimes \Fbar_2$, $\cA_4 = \cA_4 \otimes \Fbar_2$, etc. A result from \cite{ddmt} gives that $\cS_4 = \cS_4 \otimes \Fbar_2$ is irreducible. We investigate the locus of supersingular Jacobians of genus-$4$ curves, i.e., the intersection of the $9$-dimensional $\cJ_4$ with $4$-dimensional $\cS_4$ inside the $10$-dimensional $\cA_4$. The stack $\cA_4$ is smooth, and thus the codimension of an intersection in $\cA_4$ is the sum of codimensions at most; we use this in the rest of the paper without explicitly mentioning it. By considering $\F_2$-points of the loci $\cJ_4$ and $\cS_4$ and using the data from \cite{lmfdb} and \cite{xarles}, we get that $\cS_4$ is not contained in the Torelli locus $\cJ_4$. That leads to the conclusion that each irreducible component of  $$(\cS_4 \cap \cJ_4)\otimes \Fbar_2$$ is of dimension three; see Theorem \ref{thm:supersingular_curves_dim3}. In Remarks \ref{rem:eo43_implies_ss3dim} and \ref{rem:supersingular_locus_3dim}, we offer alternative proofs.

Then, we consider the Ekedahl-Oort loci and their connections with the Newton polygon strata of $\tworank$ zero and describe their intersections with $\cJ_4$. There are eight Ekedahl-Oort, and three Newton polygon strata inside the $\tworank$ zero locus $V_0$ of $\cA_4$. The mentioned Ekedahl-Oort ones $Z_{\mu}$ are defined by $$\mu \in \{[4], [4, 1], [4, 2], [4, 3], [4, 2, 1], [4, 3, 1], [4, 3, 2], [4, 3, 2, 1]\},$$ and consist of principally polarized abelian fourfolds $A$ with the type $\mu_A = \mu$.
The Newton polygon ones are defined by the slope sequences $$\left (\frac{1}{2}, \frac{1}{2}, \frac{1}{2}, \frac{1}{2}, \frac{1}{2}, \frac{1}{2}, \frac{1}{2}, \frac{1}{2}\right ), \left (\frac{1}{3}, \frac{1}{3}, \frac{1}{3}, \frac{1}{2}, \frac{1}{2}, \frac{2}{3}, \frac{2}{3}, \frac{2}{3}\right ), \text{ and } \left (\frac{1}{4}, \frac{1}{4}, \frac{1}{4}, \frac{1}{4}, \frac{3}{4}, \frac{3}{4}, \frac{3}{4}, \frac{3}{4}\right ), $$ and we denote the corresponding locally closed subsets of $\cA_4$ (consisting of principally polarized abelian varieties having those Newton polygons) with $\cS_4$,  $\mathcal{N}_{1/3}$, and $\mathcal{N}_{1/4}$. Denote $\mathcal{N}_{\leq 1/3} = \mathcal{N}_{1/3} \sqcup {\cS_4}$ and $\mathcal{N}_{\leq 1/4} = \mathcal{N}_{1/4} \sqcup \mathcal{N}_{1/3} \sqcup {\cS_4}$.

We show in Theorem \ref{thm:generic_a_num} that any generic point of $(V_0\cap \cJ_4)\otimes \Fbar_2$ has $a$-number one.
Furthermore, we show in Corollary \ref{cor:np_conclusion} and Corollary \ref{cor:eo_conclusion} that $$(\mathcal{N}_{\leq 1/3}\cap \cJ_4 )\otimes \Fbar_2 \text{ and } ( \mathcal{N}_{\leq 1/4}\cap \cJ_4 )\otimes \Fbar_2$$ are of the expected dimensions, as well as $$( Z_{\mu}\cap \cJ_4 )\otimes \Fbar_2$$ for Young types $\mu = [4], [4, 1], [4, 2]$. See \cite{zhou_genus4} for related questions in characteristic three.  Consequently, we get information about the Newton polygon and the Ekedahl-Oort stratification of $\cM_4$.

\subsection*{Acknowledgement}
The author is grateful to his supervisor Carel Faber for all the discussions and valuable comments and to Stefano Marseglia for the conversation and the help with the code supporting Example \ref{ex:isog_class_ss_bkm}. The author is supported by the Mathematical Institute of Utrecht University. 

\newpage
\section{Abelian varieties over finite fields}
Here, we recall some well-known facts about abelian varieties defined over finite fields; for example, see \cite{lmfdb_paper}.

Let $[\cA_4(\F_q)]$ denote the set of $\F_q$-isomorphism classes of principally polarized abelian fourfolds over $\F_q$ and $[\mathcal{M}_4(\F_q)]$ the set of $\F_q$-isomorphism classes of genus-$4$ smooth curves over $\F_q$. In addition, let us denote $[\mathcal{M}^0_4(\F_q)]$ and $[\mathcal{H}_4(\F_q)]$ for the sets of respectively non-hyperelliptic and hyperelliptic genus-$4$ smooth curves over $\F_q$.

Let $q = p^r$, with a prime number $p$ and $r\in \Z_{>0}$, and let $A$ be an abelian variety of dimension $g$ over $\F_q$. We denote the $q$-Frobenius with $F = F_q$ and write $$P_{A/\F_q} = P_{A/\F_q}(t) = \det (t - F| H^1(A))$$ for its characteristic polynomial (which is of degree $2g$), called \allowbreak {\textit{the Weil $q$-polynomial of $A$}.}\allowbreak 

If $B$ is another abelian variety over $\F_q$, the Honda-Tate theorem gives us that $P_{A/\F_q}$ divides $P_{B/\F_q}$ if and only if $B$ is isogenous over $\F_q$ to a product in which $A$ occurs as a factor. In particular, $$P_{A/\F_q} = P_{B/\F_q} \text{ if and only if } A \text{ and } B \text{ are isogenous over } \F_q.$$ Moreover, the characteristic polynomial $P_{A/\F_q}$ determines the Newton polygon of $A$; see \cite{lmfdb_paper}, Section 2.4.

If $P_{A/\F_q}(t) = \prod_{i = 1}^{2g}(t - \alpha_i)$ is the Weil polynomial for $A$, then the Weil polynomial associated with the base change of $A$ to $\F_{q^n}$ is 
\begin{equation}
\label{eq:base_ext_avs}
P_{A/\F_{q^n}}(t) = \prod_{i = 1}^{2g}(t - \alpha_i^n). 
\end{equation}
\noindent
That follows from the fact that the $\alpha_i$ are the eigenvalues of the $q$-Frobenius on $H^1(A)$ and the eigenvalues of the $q^n$-Frobenius are then the $\alpha_i^n$.

For a stable curve  $C$ over $\F_{q}$ of compact type, we denote with $$P_{C/\F_{q}}(t):= P_{\Jac_C/\F_{q}}(t)$$ the characteristic polynomial of $\Jac_C$ over $\F_{q}$. 

When $C$ is a smooth curve of genus $g$ defined over $\F_q$, we can explicitly determine $P_{C/\F_{q}}$ by counting the number of points of $C$ over some extensions of the base field. Consider the zeta-function of $C/\F_q$ defined by $$Z(C/\F_q, t) = \exp \left (\sum_{n\geq 1}\# C(\F_{q^n})\frac{t^n}{n} \right ). $$ The Weil conjecture for curves gives us that 
\begin{equation}
  Z(C/\F_q, t) = \frac{L(t)}{(1 - t)(1 - qt)},
  \label{eq:weil_conj}
\end{equation}
where $L(t) = t^{2g}P_{\Jac_{C}/\F_q}(1/t)$ is the $L$-polynomial of $\Jac_C/\F_q$. In particular, using the numbers $\#C(\F_{q^i})$ for $i = 1, \ldots, g$, we can determine $P_{C/\F_{q}}$, so the isogeny class of $\Jac_C$ over $\F_q$. In particular, they also determine the Newton polygon of $C$: if we write $L(t) = \sum_{i = 0}^{2g}a_it^i$ (and recall $q = p^r$), then the Newton polygon of $C$ is the lower convex hull of the points $$\{(i, \mathrm{val}_2(a_i)/r): 0\leq i\leq 2g\}$$ with $\mathrm{val_2}$ the $2$-adic valuation. $C$ is a supersingular curve if and only if its Newton polygon is a straight line with slope $1/2$ starting at $(0, 0)$ and ending at $(2g, g)$. See further \cite{pries_current_results}, Section 2.

To additionally connect the curves with the corresponding Jacobian varieties in a way that we will use in our arguments, we consider the following well-known lemma. 

\begin{lem}
\label{lem:curves_jacobians_autogp}
Let $C$ be a smooth curve over $k$. Then, the map $\epsilon \mapsto (\epsilon^{-1})^*$ gives a group homomorphism $\varphi: \Aut_{k} C \to \Aut_{k} \Jac_C$
such that: 
\begin{enumerate}
    \item $\varphi$ is injective;
    \item if $C$ is hyperelliptic, then $\varphi$ is an isomorphism;
    \item if $C$ is not hyperelliptic, then $$\Aut_{k} \Jac_C \cong \{\pm 1\}\times \Aut_{k} C.$$
\end{enumerate}
\end{lem}
\begin{proof}
This follows from Torelli's theorem \cite{milne}, Theorem 12.1; see also \cite{howe}, Proposition 4.1.
\end{proof}

\begin{rem}
The Torelli morphism induces a bijection of sets $[\mathcal{H}_4(\F_q)]$ and $[j(\mathcal{H}_4)(\F_q)]$. Furthermore, when $C$ is a non-hyperelliptic curve in $[\cM_4(\F_q)]$, then there are precisely two elements in $[\cA_4(\F_q)]$ whose representatives are isomorphic over $\Fbar_q$ to $\Jac_C$:  $\Jac_C$ and its twist $\Jac_C^{-1}$. For more details, see for example \cite{lauter}, Appendix, or \cite{berfabgeer}, Sections 8.1 and 8.2.
\end{rem}

\subsection{Computing polarizations of abelian varieties over finite fields}

Let  $h = h(x)$ be a square-free Weil polynomial without real roots and set $L = \Q[x]/(h)$. In \cite{bkm}, Bergstr\"om, Karemaker, and Marseglia gave a description of the polarizations of abelian varieties over a finite field inside the fixed isogeny class defined by $h$ under the assumption that at least one in that isogeny class admits a lifting to characteristic zero for which the reduction morphism induces an isomorphism of endomorphism rings. That description relies on the antiequivalence of categories 
\begin{equation}
\mathcal{G}: \mathrm{AV}_{\F_p}(h) \to \mathcal{I}(R_h),
\label{eq:CS_equiv}
\end{equation}
where $\mathrm{AV}_{\F_p}(h)$ is the category of abelian varieties over $\F_p$ inside the $\F_p$-isogeny class defined by the polynomial $h$ (using Honda-Tate theory) and with $\F_p$-morphisms, and $\mathcal{I}(R_h)$ is the category of fractional $R_h$-ideals in $L$ with $R_h$-linear morphisms, where $R_h$ is the order in $L$ generated by $\pi = x \text{ }(\text{mod }h)$ and $p/\pi$. The equivalence is provided by \cite{mars}, Theorem 4.3, and obtained using the Centeleghe-Stix equivalence from \cite{centstix}. 

Under $\mathcal{G}$, for $B, B' \in \mathrm{AV}_{\F_p}(h)$ we have $\mathcal{G}(\Hom(B, B')) = (\mathcal{G}(B): \mathcal{G}(B'))$, and if $\mathcal{G}(B) = I$ then $\mathcal{G}(B^{\vee}) = \Bar{I}^t$, where $I^t = \{a\in L: \mathrm{Tr}_{L/\Q}(aI) \subseteq \Z\}$ is the trace dual ideal of $I$ and $\overline{I} = \{\bar{x}: x\in I\}$ with $x\mapsto \bar{x}$ the involution of $L$. 

Let $L_{\mathbb{R}}$ be the unique totally real subalgebra of $L$, let $L^{+}$ be the set of totally positive elements of $L_{\mathbb{R}}$, and let $S_{\mathbb{R}}^* = S^*\cap L_{\mathbb{R}}$ be the group of totally real units in $S$. Furthermore, for a CM-type $\Phi$ of $L$, let 
$$\Sigma_{\Phi} = \begin{Bmatrix}
 & S \text{ is Gorenstein and } S = \Bar{S}, \text{ and there exists }  \\ S \subseteq L \text{ order}: &\text{ } A_0\in \mathrm{AV}_{\F_p}(h) \text{ with CM-type }\Phi,  \text{ and } \mathrm{End}(A_0) = S, \\ & \text{ that admits a canonical lifting to a }p\text{-adic field }K
\end{Bmatrix}.$$

 Consider an abelian variety $B_0$ in $\mathrm{AV}_{\F_p}(h)$, and denote $I = \mathcal{G}(B_0)$, $T = (I:I)$. Let $\mathcal{T}$ be a transversal of the quotient $T^*/\<v\bar{v}: v \in T^*\>$, i.e., a set containing exactly one representative for each class in $T^*/\<v\bar{v}: v \in T^*\>$. If $B_0$ is principally polarized, then, in particular, it follows that $B_0 \cong B_0^{\vee}$, so there is an element $i_0\in L^*$ such that $i_0 \bar{I}^t = I$. If such $i_0$ exists, for an arbitrary $\CC$-valued CM-type, we denote  $$\mathcal{P}_{\Phi}^{1}(I) = \{i_0u: u \in \mathcal{T} \text{ such that }i_0u \text{ is totally imaginary and }\Phi\text{-positive}\}. $$ Otherwise, when $B_0$ is not isomorphic to $B_0^{\vee}$, such $i_0$ does not exist, and we  denote $$\mathcal{P}_{\Phi}^{1}(I) = \o,$$ for all $\Phi$. 

We collect some results of \cite{bkm} into the following Proposition. The authors of \cite{bkm} offer an algorithm for computing representatives of principal polarizations of any abelian variety inside a fixed isogeny class whenever \eqref{eqthm:552} is satisfied, using that \eqref{eqthm:552} can be checked algorithmically. We use this in {Section \ref{sec:jacobians} to prove Theorem \ref{thm:supersingular_curves_dim3}.

\begin{prop} Let $\Phi_b$ be a CM-type for $L = \Q[x]/(h)$ for which $\Sigma_b$ is non-empty and let $S\in \Sigma_{\Phi_b}$ be an order. For  an abelian variety $B_0 \in \mathrm{AV}_{\F_p}(h)$ let $I = \mathcal{G}(B_0)$ and $T = (I:I)$, and let $\mathcal{T}$ be a transversal of $T^*/\<v\bar{v}: v\in T^*\>$ as above. If it holds that 
\begin{equation}
\text{for every }\xi \in S_{\mathbb{R}}^* \text{ we have }\xi L^{+}\cap \mathcal{T} \neq \o,
\label{eqthm:552}
\end{equation}
then the complete set of representatives of principal polarizations of $B_0$ up to isomorphism is in bijection with the set $\mathcal{P}^{1}_{\Phi_b}(I)$.
\label{thm:552}
\end{prop}

\begin{proof}
This follows from \cite{bkm}, Theorem 5.5,  Theorem 4.10 and Proposition 5.2.
\end{proof}

\begin{rem}
\label{rmk:reflex_cond}
\cite{bkm}, Corollary 4.9 gives that $\Sigma_{\Phi}$ is non-empty under the assumption that $(L, \Phi)$ satisfies the \textit{generalized residue reflex condition} (see \cite{bkm}, Definition 2.14). 
\end{rem}

\newpage
\section{Jacobian varieties over $\F_2$}
\label{sec:jacobians}

In this section, we describe the Jacobians of genus-$4$ curves of compact type over $\F_2$. By comparing the stack counts in certain $\F_2$-isogeny classes of principally polarized abelian varieties with the ones coming from Jacobians, i.e., by comparing $\F_2$-points of $\cS_4 = \cS_4 \otimes \Fbar_2$ and $\cJ_4 = \cJ_4 \otimes \Fbar_2$,  we will conclude that the supersingular locus $\cS_4$ is not contained in $\cJ_4$.    

Let $C$ be a stable genus-$4$ curve of compact type, defined over $\F_2$. It is well-known that the Jacobian of $C$ is isomorphic to the product of Jacobians of its (non-rational) irreducible components. Write $C_{g, \F_{2^n}}^{(i)}$ for a smooth curve of genus $g\geq 1$ defined over $\F_{2^n}$, which is an irreducible component of $C$; the $(i)$-notation is to denote, in general, distinct curves. Up to isomorphism over $\F_2$, we have the following list of possibilities for a Jacobian of a genus-$4$ curve of compact type defined over $\F_2$.

\begin{enumerate}[(i)]
    \item $\Jac_C = \Jac_{C_{4, \F_2}^{(0)}}$.
    \item $\Jac_C = \Jac_{C_{3, \F_2}^{(0)}} \times\Jac_{C_{1, \F_2}^{(0)}}$.
    \item $\Jac_C = \Jac_{C_{2, \F_2}^{(1)}}\times\Jac_{C_{2, \F_2}^{(2)}}$.
    \item $\Jac_C = \Jac_{C_{1, \F_2}^{(1)}}\times\Jac_{C_{1, \F_2}^{(2)}}\times \Jac_{C_{1, \F_2}^{(3)}}\times\Jac_{C_{1, \F_2}^{(4)}}$.
    \item $\Jac_C = \Jac_{C_{2, \F_2}^{(0)}}\times\Jac_{C_{1, \F_2}^{(1)}}\times\Jac_{C_{1, \F_2}^{(2)}}$.
    \item $\Jac_C = \Jac_{C_{2, \F_4}^{(1)}}\times\Jac_{C_{2, \F_4}^{(2)}}$, where $C_{2, \F_4}^{(1)}$ and $C_{2, \F_4}^{(2)}$ are sent one to another by the Frobenius, i.e., they are conjugated.
    \item $\Jac_C = \Jac_{C_{2, \F_2}^{(0)}}\times\Jac_{C_{1, \F_4}^{(1)}}\times\Jac_{C_{1, \F_4}^{(2)}}$, where $C_{1, \F_4}^{(1)}$ and $C_{1, \F_4}^{(2)}$ are conjugated.
    \item $\Jac_C = \Jac_{C_{1, \F_2}^{(1)}}\times\Jac_{C_{1, \F_2}^{(2)}}\times\Jac_{C_{1, \F_4}^{(3)}}\times\Jac_{C_{1, \F_4}^{(4)}}$, where $C_{1, \F_4}^{(3)}$ and $C_{1, \F_4}^{(4)}$ are conjugated.
    \item $\Jac_C = \Jac_{C_{1, \F_4}^{(1)}}\times\Jac_{C_{1, \F_4}^{(2)}}\times\Jac_{C_{1, \F_4}^{(3)}}\times\Jac_{C_{1, \F_4}^{(4)}}$, where, without loss of generality, ${C_{1, \F_4}^{(1)}}$ and ${C_{1, \F_4}^{(2)}}$, as well as ${C_{1, \F_4}^{(3)}}$ and ${C_{1, \F_4}^{(4)}}$, are two pairs of mutually conjugated curves.
    \item $\Jac_C = \Jac_{C_{1, \F_2}^{(0)}}\times\Jac_{C_{1, \F_8}^{(1)}}\times\Jac_{C_{1, \F_8}^{(2)}}\times\Jac_{C_{1, \F_8}^{(3)}}$, where ${C_{1, \F_8}^{(1)}}, {C_{1, \F_8}^{(2)}}$, and ${C_{1, \F_8}^{(3)}}$ are three irreducible components of $C$ defined over $\F_8$, cyclically permuted by the Frobenius morphism.
    \item $\Jac_C = \Jac_{C_{1, \F_{16}}^{(1)}}\times\Jac_{C_{1, \F_{16}}^{(2)}}\times\Jac_{C_{1, \F_{16}}^{(3)}}\times\Jac_{C_{1, \F_{16}}^{(4)}}$, where all $C_{1, \F_{16}}^{(1)}, C_{1, \F_{16}}^{(2)}, C_{1, \F_{16}}^{(3)}$, and $C_{1, \F_{16}}^{(4)}$ are irreducible components of $C$ defined over $\F_{16}$, cyclically permuted by the Frobenius.  
\end{enumerate}

Let $C$ be a singular, stable genus-$4$ curve defined over $\F_2$, whose Jacobian variety has one of the types from (vi)-(xi), and let $D, D'$ be two components of $C$ defined over $\F_{2^n}$ for $n \in \{2, 3, 4\}$ so that $D$ and $D'$ are conjugated by Frobenius. Since finite fields are perfect, we get that $D$ and $D'$ have the same number of $\F_{2^n}$- and $\F_{2^{2n}}$-points. In particular, they are in the same $\F_{2^n}$-isogeny class, so we have that $(P_{D/\F_{2^n}})^2$, the square of the characteristic polynomial of $\Jac_D$ over $\F_{2^n}$, divides $P_{C/\F_{2^n}}$, the characteristic polynomial of $\Jac_C$ over $\F_{2^n}$.
\\

\noindent
From the database \cite{lmfdb}, we extracted all ($65$) $\F_2$-isogeny classes of supersingular abelian fourfolds. In \cite{xarles}, Xarles determined all smooth genus-$4$ curves over $\F_2$ up to $\F_2$-isomorphism. We used the provided list of curves and their numbers of $\F_{2^r}$-points, for $r = 1, 2, 3, 4$, to deduce which are supersingular and to determine the $\F_2$-isogeny classes of their Jacobians. In addition, for each supersingular genus-$4$ curve $C$ over $\F_2$, we computed the size of $\Aut_{\F_2}(C)$. 

First, we give an example of a smooth, genus-$4$ curve over $\F_2$, which is supersingular.

\begin{exmp}
\label{ex:smooth_ss_curve}
Consider the genus-$4$ curve $C$ defined over $\F_2$ by the equations  $$C:\left\{\begin{matrix}
XY + ZT = 0\\ 
X^2Z + Y^2Z + YZ^2 + X^2T + Y^2T + XT^2 = 0
\end{matrix}\right. \text{ in } \P^3.$$ It can be computed that $\#C(\F_2) = 7, \#C(\F_4) =9, \#C(\F_8) = 13$, and $\#C(\F_{16}) = 9$, and using \eqref{eq:weil_conj}, one obtains $P_{C/\F_2} = t^8 + 4t^7 + 10t^6 + 20t^5 + 32t^4 + 40t^3 + 40t^2 + 32t + 16.$ It follows that $C$ is a supersingular curve.
\end{exmp}

In the following two examples, we consider an $\F_2$-isogeny class of abelian fourfolds and describe its elements.

\begin{exmp}
Consider the $\F_2$-isogeny class $[A]$ of abelian varieties of dimension $4$ defined by the characteristic polynomial  $$P_{A/\F_2} = t^8 + 2t^7 + 2t^6 - 4t^4 + 8t^2 + 16t + 16;$$ it is an irreducible (and therefore square-free) polynomial without real roots. Note that all abelian varieties inside $[A]$ are supersingular. Using \eqref{eq:base_ext_avs}, we find that the isogeny classes of the base extensions of abelian varieties in $[A]$ to $\F_4$ and $\F_8$ are defined by irreducible polynomials $P_{A/\F_4}$ and $P_{A/\F_8}$, while the $\F_{16}$-isogeny class of the base extension of $[A]$ to $\F_{16}$ is defined by $$P_{A/\F_{16}} = (t^4 - 4t^3 + 16t^2 - 64t + 256)^2 = (P_{S/\F_{16}})^2$$ with $P_{S/\F_{16}}$ irreducible. 

The only possibility that $\Jac_C$ is in $[A]$, with $C$ a stable, genus-$4$ curve of compact type defined over $\F_2$, is, in fact, when $C$ is a smooth curve. Indeed, if there are at least two irreducible components of $C$ defined over $\F_{2^n}$ with $n \in \{1, 2, 3\}$, the polynomial $P_{C/\F_{2^n}}$ would have at least two irreducible factors. Similarly, if $C$ has four components $C^{(i)}_{1, \F_{16}}$, $i = 1, 2, 3, 4$ of genus $1$, then $P_{C/\F_{16}} =  (P_{C^{(1)}_{1, \F_{16}}} )^4$. None of that is possible, so $C$ is necessarily smooth.

Using Xarles's data, we find only one smooth genus-$4$ curve over $\F_2$, whose Jacobian is in $[A]$. Namely, it is a hyperelliptic curve defined by the equation $$C: y^2 + y = x^9 + x^5, $$ for which we computed $|\Aut_{\F_2}(C)| = 4$. Lemma \ref{lem:curves_jacobians_autogp} gives that $[\Jac_C]$ is an $\F_2$-point of $\cA_4$ weighted by $\frac{1}{4}$ lying inside the isogeny class $[A]$. Therefore, we got $$\sum_{\substack{\Jac_C \in [\cJ_4(\F_2)] \\ \Jac_C \text{ in } [A]}} \frac{1}{|\Aut_{\F_2}(\Jac_C)|} = \frac{1}{4}.$$

\label{ex:isog_class_ss}
\end{exmp}

\begin{exmp}
Consider the same $\F_2$-isogeny class $[A]$ as in Example \ref{ex:isog_class_ss}, and write $h = P_{A/\F_2}$. The outcome of a Magma code gives that there is a CM-type $\Phi$ of endomorphism algebra $L = \Q[x]/h$, so that $(L, \Phi)$ satisfies the generalized residue reflex condition. By Remark \ref{rmk:reflex_cond}, it follows that $\Sigma_{\Phi}$ is non-empty. 

We find six overorders $T_i$ of $R_h$ inside $\O_L$, of index $i$ in $\O_L$ for $i = 1, 4, 8, 16, 32, 64$, and all satisfy condition \eqref{eqthm:552} of Proposition \ref{thm:552}. Hence, for all $i \in \{1, 4, 8, 16, 32, 64\}$ and for each fractional ideal $I$ with $(I:I) = T_i$, we can compute $|\mathcal{P}_{\Phi}^1(I)|$. That gives us a description of polarized abelian varieties $B$ in $\mathrm{AV}_{\F_2}(h)$ corresponding to fractional ideals $I$  under the antiequivalence of categories \eqref{eq:CS_equiv}; moreover, for each such $B$, we can compute the number of its (polarized) automorphisms over $\F_2$. 

Write $I_B = \mathcal{G}(B)$, for the image of  $B$ in $\mathrm{AV}_{\F_2}(h)$ under \eqref{eq:CS_equiv}, and let $(I_B:I_B) = T_i$. The obtained description of all $B\in \mathrm{AV}_{\F_2}(h)$ in terms of $i$ follows.
\begin{itemize}
    \item For $i = 1, 16, 32$, none of the abelian varieties $B$ in $\mathrm{AV}_{\F_2}(h)$ is principally polarizable.
    \item For $i = 4$, there is a unique principally polarized abelian variety $B$ in $\mathrm{AV}_{\F_2}(h)$ and $|\Aut_{\F_2}(B)| = 4$.
    \item For $i = 8$, there are two principally polarized abelian varieties $B$ in $\mathrm{AV}_{\F_2}(h)$, both with $|\Aut_{\F_2}(B)| = 4$.
    \item For $i = 64$, there is a single unpolarized abelian variety in $\mathrm{AV}_{\F_2}(h)$, and it possesses two non-isomorphic principal polarizations. In other words, here we get two non-isomorphic principally polarized varieties $B$. For both of them, it holds that $|\Aut_{\F_2}(B)| = 2$. 
\end{itemize}
 
In particular, the stack count we get here is $$\sum_{\substack{B \in [\mathcal{A}_4(\F_2)] \\ B \text{ in } [A]}} \frac{1}{|\Aut_{\F_2}(B)|} = \frac{1}{4} + 2\cdot \frac{1}{4} + 2\cdot \frac{1}{2} = \frac{7}{4}.$$

\label{ex:isog_class_ss_bkm}
\end{exmp}

\begin{exmp} We obtain the same numerical outcome in terms of the stack counts as in Example \ref{ex:isog_class_ss_bkm} and Example \ref{ex:isog_class_ss}, when we take $$h = P_{A/\F_2} = t^8 - 2t^7 + 2t^6 - 4t^4 + 8t^2 - 16t + 16.$$ In that case, we find only one stable, genus-$4$ curve $C$ of compact type defined over $\F_2$, for which $\Jac_C$ is inside the $\F_2$-isogeny class defined by $h$. Namely, we have that $C$ is the smooth hyperelliptic curve defined over $\F_2$ by the standard equation $$y^2 + y = x^9 + x^5 + 1.$$ It is a twist of the curve considered in Example  \ref{ex:isog_class_ss}.
\end{exmp}

Let now $\cA_4 = \cA_4 \otimes \Fbar_2$ and $\cJ_4 = \cJ_4 \otimes \Fbar_2$, and let $\cS_4$, $\mathcal{N}_{\leq 1/3}$, and $\mathcal{N}_{\leq 1/4}$ be the Newton polygon strata inside $\cA_4$ in characteristic two.
First, recall the result earlier obtained. 
\begin{lem}
\label{lem:ss4}
In characteristic two, the supersingular locus $\mathcal{S}_4$ is irreducible in  $\mathcal{A}_4$.
\end{lem}
\begin{proof}
See \cite{ddmt}, Theorem 3.5.
\end{proof}

Collecting the previous observations, we get the following result. 

{\begin{thm} In characteristic two, the supersingular locus $\cS_4\cap \cJ_4$ of Jacobians is pure of dimension three.
\label{thm:supersingular_curves_dim3}
\end{thm}

\begin{proof}
Let $\mathcal{W}$ be an irreducible component of $\cS_4\cap \cJ_4$. Recall that $\cS_4 = \cS_4\otimes \Fbar_2$ is irreducible by Lemma \ref{lem:ss4}, proper, and has dimension $\lfloor 4^2/4 \rfloor = 4$, and that $\cJ_4$ is closed of codimension $1$ in $\cA_4$. We find by Example \ref{ex:smooth_ss_curve} that  $\cS_4\cap \cJ_4$ is non-empty, and thus, the dimension of each component of $\cS_4\cap \cJ_4$ is either $3$ or $4$. 

Assume that $\dim \mathcal{W} = 4$. Then, $\mathcal{W}$ coincides with $\mathcal{S}_4$ and is contained in the Torelli locus $\mathcal{J}_4$. However, by Examples \ref{ex:isog_class_ss} and \ref{ex:isog_class_ss_bkm}, that is not possible. Namely, the inclusion of $\mathcal{S}_4$ in $\mathcal{J}_4$ would result in the inclusion of their $\F_2$-points; that would imply that we should get the same stack counts in those examples. Therefore, $\dim \mathcal{W} = 3$.
\end{proof}}

\begin{rem}
There is a recent result by Harashita describing certain irreducible components $\mathcal{W}$ of $(\cJ_4\cap \cS_4)\otimes \Fbar_p$. Namely, if $\mathcal{W}$ contains an image of a non-hyperelliptic superspecial (smooth) curve, then by \cite{harashita21}, Theorem 2.6, $\mathcal{W}$ is three-dimensional. Note that we cannot use it for $p = 2$ since there are no superspecial genus-$g$ curves in characteristic two for $g\geq 2$.  
\end{rem}

\begin{cor} In characteristic two, the Newton polygon strata $\mathcal{N}_{\leq 1/3}\cap \cJ_4$ and $\mathcal{N}_{\leq 1/4}\cap \cJ_4$  are of the (expected) codimensions $4$ and $5$ inside $\cJ_4$.
\label{cor:np_conclusion}
\end{cor}

\begin{proof}
For $\mathcal{N}_{\leq 1/4}\cap \cJ_4$, the result follows from \cite{fabervdgeer}, Theorem 2.3. 

For the second claim, consider for example, the curve $C$ given by the equations $$C:\left\{\begin{matrix}
XY + ZT = 0\\ 
T^2X + TX^2 + Y^3 + X^2Z + XZ^2 = 0
\end{matrix}\right. \text{ in } \P^3.  $$ It lies in $\mathcal{N}_{1/3}$ since $\#C(\F_2) = 5, \#C(\F_4) = 9, \#C(\F_8) = 11$, and $\#C(\F_{16}) = 17$. Furthermore,  $\cS_4\subseteq \mathcal{N}_{\leq 1/3}$, and thus, $\mathcal{N}_{\leq 1/3}\cap \cJ_4$ is not fully contained in the Torelli locus, so the result for $\mathcal{N}_{\leq 1/3}\cap \cJ_4$ follows. 
\end{proof}

\newpage
\label{sec:Oort}
\section{Ekedahl-Oort stratification for $g = 4$} 
In this section, we mention two classification results about the Ekedahl-Oort stratification of $\cA_4 = \cA_4 \otimes \Fbar_p$. Using them, we will connect certain Ekedahl-Oort types and Newton polygons and describe the stable, genus-$4$ curves of compact type lying inside the loci determined by those invariants.

Ekedahl-Oort strata in $\cA_g$ are defined by the classification of the $p$-torsion group schemes $A[p]$ of abelian varieties $A$ up to isomorphism (of $\mathrm{BT}_1$-group schemes), while the Newton polygon strata are defined by the classification of the $p$-divisible groups $A[p^{\infty}]$ of abelian varieties $A$ up to isogeny (of $p$-divisible groups). Recall the well-known facts about these stratifications.

\begin{itemize}
    \item The Ekedahl-Oort strata in $\cA_4$ of $p$-rank zero are precisely $Z_{\mu}$ for those $\mu$ appearing in the diagram below,  
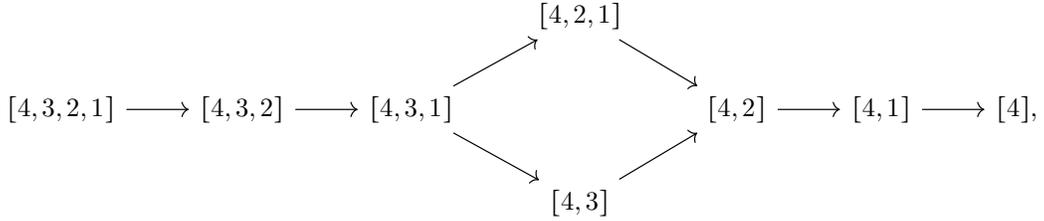
\begin{small}
\begin{center}
\begin{tikzcd}
	&&& {[4, 2, 1]} \\
	{[4, 3, 2, 1]} & {[4, 3, 2]} & {[4, 3, 1]} && {[4, 2]} & {[4, 1]} & {[4]}, \\
	&&& {[4, 3]}
	\arrow[from=2-1, to=2-2]
	\arrow[from=2-2, to=2-3]
	\arrow[from=2-3, to=1-4]
	\arrow[from=2-3, to=3-4]
	\arrow[from=3-4, to=2-5]
	\arrow[from=1-4, to=2-5]
	\arrow[from=2-5, to=2-6]
	\arrow[from=2-6, to=2-7]
\end{tikzcd}  
\end{center}
\end{small}
where there is an arrow $\mu_1 \to \mu_2$ if $Z_{\mu_1} \subseteq \overline{Z_{\mu_2}}$. For $\mu = [\mu_1, \mu_2, \ldots, \mu_n]$, the codimension of $Z_{\mu}$ in $\cA_4$ equals $\sum_{i = 1}^n\mu_i$.
\item The Newton polygon strata inside the $p$-rank zero locus are the ones defined by the slope sequences $$\left (\frac{1}{2}, \frac{1}{2}, \frac{1}{2}, \frac{1}{2}, \frac{1}{2}, \frac{1}{2}, \frac{1}{2}, \frac{1}{2}\right ), \left (\frac{1}{3}, \frac{1}{3}, \frac{1}{3}, \frac{1}{2}, \frac{1}{2}, \frac{2}{3}, \frac{2}{3}, \frac{2}{3}\right ), \text{ and } \left (\frac{1}{4}, \frac{1}{4}, \frac{1}{4}, \frac{1}{4}, \frac{3}{4}, \frac{3}{4}, \frac{3}{4}, \frac{3}{4}\right ).$$ We denote these strata respectively by $\mathcal{S}_{4}$, $\mathcal{N}_{1/3}$, and $\mathcal{N}_{1/4}$.
\end{itemize}

In \cite{ibukiyamakaremakeryu}, Ibukiyama, Karemaker, and Yu obtain the following result, which tells us how the mentioned Ekedahl-Oort and Newton polygon loci are related.

\begin{prop} Let $p$ be a prime number and $\cA_4 = \cA_4 \otimes \Fbar_p$ the moduli space of principally polarized abelian varieties in characteristic $p$. The following results hold. 

\begin{enumerate}
    \item The strata $Z_{\mu}$, for $\mu \in \{[4, 3], [4, 3, 1], [4, 3, 2], [4, 3, 2, 1]\}$, are completely contained in $\cS_4$, and we have 
    \begin{equation*}
    \cS_4 = (Z_{[4]}\cap \cS_4) \sqcup (Z_{[4, 2]}\cap \cS_4) \sqcup Z_{[4, 3]}\sqcup Z_{[4, 3, 1]}\sqcup Z_{[4, 3, 2]}\sqcup Z_{[4, 3, 2, 1]}; 
    \end{equation*} $Z_{[4]}\cap \cS_4$ is dense in $\cS_4$.

    \item The stratum $Z_{[4, 2, 1]}$ is completely contained in $\mathcal{N}_{1/3}$, and we have $$\mathcal{N}_{1/3} = (Z_{[4]}\cap \mathcal{N}_{1/3} )\sqcup (Z_{[4, 2]}\cap \mathcal{N}_{1/3} ) \sqcup Z_{[4, 2, 1]}; $$  $(Z_{[4]}\cap \mathcal{N}_{1/3} )$ is dense in $ \mathcal{N}_{1/3}$.
    
    \item The stratum $Z_{[4, 1]}$ is completely contained in $\mathcal{N}_{1/4}$, and we have $$\mathcal{N}_{1/4} = (Z_{[4]}\cap \mathcal{N}_{1/4} )\sqcup Z_{[4, 1]}; $$  $(Z_{[4]}\cap \mathcal{N}_{1/4} )$ is dense in $ \mathcal{N}_{1/4}$.
\end{enumerate}
\label{prop:first_eo_prop}
\end{prop}
\begin{proof}
See \cite{ibukiyamakaremakeryu}, Proposition 5.13.
\end{proof}

\begin{rem}
By Proposition \ref{prop:first_eo_prop}, we have $Z_{[4, 3]}\subseteq \cS_4$, $Z_{[4, 2]}\subseteq \cS_4 \cup \mathcal{N}_{1/3}$, and $Z_{[4, 1]}\subseteq \mathcal{N}_{1/4}$. Therefore, for an abelian variety $A$, we have

\begin{equation*} a(A) = 2 \implies
\begin{cases}
A \in Z_{[4, 1]} & \text{ if } A \in \mathcal{N}_{1/4} \\ 
A \in Z_{[4, 2]} & \text{ if } A \in \mathcal{N}_{1/3} \\
A \in Z_{[4, 2]}\cup Z_{[4, 3]} & \text{ if } A \in \mathcal{S}_{4}
\end{cases}.
\end{equation*}
\label{rem:a_num_2_eo_np}
\end{rem}

Let $p = 2$. Using the characterization of the space of regular differentials in terms of theta characteristics by Stöhr and Voloch in \cite{stohrvoloch}, Proposition 3.1, for non-hyperelliptic, genus-$4$ smooth curves $C$, which lie on a quadric cone, we find that
\begin{equation}
\tworank(C) = 0 \implies  a(C) = 2.
\label{eq:anumberofcone}    
\end{equation}
Let us now combine that result with the previous remark.

\begin{exmp} As an immediate application of the previous results, we find a smooth genus-$4$ curve over $\F_2$ with the Ekedahl-Oort type $[4, 1]$. Consider the curve 
$$C:\left\{\begin{matrix}
XY + T^2 = 0\\ 
TX^2 + Y^3 + X^2Z + Z^3 = 0
\end{matrix}\right. \text{ in } \P^3.$$ Using Xarles's data, we find $$\#C(\F_2) = 5, \#C(\F_4) = 9, \#C(\F_8) = 11, \text{ and } \#C(\F_{16}) = 25, $$ and thus, we see that $\Jac_C \in \mathcal{N}_{1/4}$, so, in particular, $\tworank(C) = 0$. Furthermore, $C$ is a smooth curve lying on a quadric cone, so we conclude by the previous $a(\Jac_C) = a(C) = 2$. Now Remark \ref{rem:a_num_2_eo_np} gives us that $$\Jac_C \in Z_{[4, 1]}.$$
\label{exmp:EO41nonempty}
\end{exmp}

\begin{exmp} Xarles's data from \cite{xarles} say there are no smooth genus-$4$ curves $C$ over $\F_2$ lying on a quadric cone, such that $\Jac_C$ is in $\mathcal{N}_{1/3}\cup \cS_4$. 
\end{exmp}

Now we offer a characterization of the Ekedahl-Oort types in terms of indecomposable $\mathrm{BT}_1$ group schemes. Recall the construction from \cite{oort_eo_classification} by Oort. 

Let $$\nu = \{\nu(1),\nu(2), \ldots, \nu(g) \}, \text{ }\text{ with }\text{ } \nu(2g - i) = \nu(i) + g - i,\text{ } 1\leq i \leq g, \text{ }\text{and } \nu(0) = 0$$ be the final type of a principally polarized abelian variety $A$ of dimension $g$ over $k$, let $$1\leq m_1 < m_2 < \ldots < m_g\leq 2g$$ be those $i$ with $\nu(i - 1)< \nu(i)$, and $$1\leq n_g < n_{g - 1} < \ldots < n_1\leq 2g$$ be the elements in the complement $\{1, 2, \ldots, 2g\} - \{m_1, m_2, \ldots, m_g\}$. By \cite{oort_eo_classification}, Theorem 9.4, the polarized covariant Dieudonn\'e module $\mathbb{D}(A[p])$ is isomorphic to the one with the basis $$\{Z_{1}, Z_{2}, \ldots, Z_{2g}\}, 
\text{  where  } Z_{m_i} = X_i \text{ and } Z_{n_i} = Y_i, \text{ for } i \in \{1, 2, \ldots g\}$$ and the relations $$\mathcal{F}(Z_{m_i}) = Z_{i}, \quad \mathcal{F}(Z_{n_i}) =  0, \quad \mathcal{V}(Z_{i}) = 0, \text{  and   } \mathcal{V}(Z_{2g - i + 1}) = \begin{cases}
Z_{n_i} & \text{if } 2g - i + 1 \in \{n_1, \ldots, n_g\} \\ 
-Z_{n_i} & \text{otherwise }
\end{cases},$$ for any $i \in \{1, 2, \ldots, g\}$. The quasi-polarization is given by $$\<X_i, Y_j\> = \delta_{ij}, \quad \<X_i, X_j\> = 0, \quad \text{and}\quad \<Y_i, Y_j\> = 0.$$

In the following examples, we will describe some $\mathrm{BT}_1$ group schemes $A[p]$ of abelian varieties $A$ inside the $\prank$ zero loci. 

\begin{lem}
For any $g\geq 1$, there is a unique, indecomposable $\mathrm{BT}_1$ group scheme $I_{g, 1}$ of rank $p^{2g}$, $p$-rank $0$, and $a$-number $1$. In other words, for an abelian variety $A$ with $A[p]\cong I_{g, 1}$, we have $\mu_A = [g]$, i.e., $\nu_A = \{0, 1, 2, \ldots, g-1\}$.
\end{lem}
\begin{proof}
 The proof follows from the fact that $\mathcal{V}$ is nilpotent on $I_{g, 1}$ and that $\mu_{A, 2} = 0$ uniquely determines $I_{g, 1}$; see \cite{pries_eo_classification}, Lemma 3.1. 
 The $\mathrm{BT}_1$ group scheme $I_{g, 1}$ is necessarily indecomposable as $\prank$ and $a$-number are additive functions. 

Let us briefly describe the basis $\{Z_1, \ldots, Z_{2g}\}$ of $\mathbb{D}(I_{g, 1})$ and how it behaves under $\mathcal{F}$ and $\mathcal{V}$. We have $$Z_{2} = X_{1}, \ldots,Z_{g-1} = X_g, Z_{2g} = X_{g}, \quad \text{and}\quad Z_{1} = Y_g, Z_{g + 1} = Y_{g - 1}, \ldots Z_{2g - 1} = Y_1, $$ with 
$$\mathcal{F}(X_i) = \begin{cases}
X_{i - 1} & \text{ if } 2\leq i \leq g \\ 
Y_g & \text{ if } i = 1
\end{cases},\quad \text{and}\quad \mathcal{F}(Y_i) = 0 \text{ for }i\in \{1, \ldots, g\},$$ and $$\mathcal{V}(Y_i) = \begin{cases}
Y_{i + 1} & \text{ if } 1\leq i \leq g-1 \\ 
0 & \text{ if } i = g
\end{cases},\quad\text{and}\quad \mathcal{V}(X_i) = \begin{cases}
0 & \text{ if } 1\leq i \leq g-1 \\ 
-Y_1 & \text{ if } i = g
\end{cases}.$$
\label{lem:Ig1}
\end{proof}

\begin{exmp}[Dimension 1]
In dimension $1$, the notions of being inside the $\prank$ zero locus, being supersingular, superspecial, and having the Ekedahl-Oort type $[1]$ coincide. By Lemma \ref{lem:Ig1}, an abelian variety $A$ having any of those properties satisfies $A[p]\cong I_{1, 1}$. 
\end{exmp}

\begin{exmp}[Dimension 2]
There are two Ekedahl-Oort types of abelian surfaces $A$ inside the $\prank$ zero locus. It immediately follows that 
\begin{itemize}
    \item either $\mu_A = [2]$, when $A[p]\cong I_{2, 1}$ by Lemma \ref{lem:Ig1};
    \item or $\mu_A = [2, 1]$, when $A$ is superspecial and $A[p]\cong I_{1, 1}^{\oplus 2}$.
\end{itemize}
\end{exmp}

\begin{exmp}[Dimension 3]
The Ekedahl-Oort types occurring for abelian threefolds $A$ inside the $\prank$ zero locus are $[3], [3, 1], [3, 2]$, and $[3, 2, 1]$. 
\begin{itemize}
    \item If $\mu_A = [3]$, then $A[p] \cong I_{3, 1}$ as explained in Lemma \ref{lem:Ig1}.
    \item When $\mu_A = [3, 1]$, then $A[p]\cong I_{3,2}$ is an indecomposable $\mathrm{BT}_1$ group scheme; see Example \ref{exmp:dim4_eo_anum3} for the indecomposability argument. 
    
    Its covariant Dieudonn\'e module $D_6 = \mathbb{D}(I_{3, 2})$ has a basis $\{Y_3, X_1, Y_2, X_2, Y_1, X_3\}$, and $$\mathcal{F}: \bigl(\begin{smallmatrix}
Y_3& X_1& Y_2& X_2& Y_1& X_3 \\ 
0 & Y_3 &0  & X_1  &0  & Y_2 
\end{smallmatrix}\bigr) \text{ and } \mathcal{V}: \bigl(\begin{smallmatrix}
Y_3& X_1& Y_2& X_2& Y_1& X_3 \\ 
0 & 0 &0  & -Y_3  &Y_2  & -Y_1 
\end{smallmatrix}\bigr).$$ Therefore, we get $\mathcal{V}(D_6) = \<Y_3, Y_2, Y_1\> = D_3$, and $\mathcal{V}^2(D_6) = \<Y_2\> = D_1$. Then $\mathcal{F}^{-1}(D_1) = \<Y_3, Y_2, Y_1, X_3\> = D_4$, so $\mathcal{V}(D_4) = \<Y_2, Y_1\> = D_2$ and $\mathcal{V}(D_2) = \<Y_2\> = D_1$. From this, we clearly see that the final type of $A$ is $\nu_A = \{0, 1, 1\}$.

    \item If $\mu_A = [3, 2]$, then $A[p]\cong I_{2, 1}\oplus I_{1, 1}$. To see this, use Lemma \ref{lem:Ig1} to compute the Ekedahl-Oort type of $I_{2, 1}\oplus I_{1, 1}$.
    \item Lastly, if $\mu_A = [3, 2, 1]$, then $A$ is superspecial and $A[p]\cong I_{1, 1}^{\oplus 3}$.
\end{itemize}
\label{exmp:dim3_eo_anum2}
\end{exmp}

Lastly, we classify the Ekedahl-Oort types in dimension four, of those abelian varieties having the $a$-number three and lying inside the $\prank$ zero locus.

\begin{exmp}
There are three Ekedahl-Oort types of abelian fourfolds with $\prank$ zero and $a$-number three. Namely, $\mu$ can be $$[4, 2, 1], [4, 3, 1], \text{ or }[4, 3, 2].$$ Using the additivity of the $\prank$ and the $a$-number, we see that two of them define decomposable $\mathrm{BT}_1$ group schemes, and consequently, the remaining one defines an indecomposable one. 

Consider first $A[p]\cong I_{3, 2}\oplus I_{1, 1}$. Lemma \ref{lem:Ig1} and Example \ref{exmp:dim3_eo_anum2} give us that $$D_8 = \mathbb{D}(I_{3, 2}\oplus I_{1, 1}) = \mathbb{D}(I_{3, 2})\oplus \mathbb{D}(I_{1, 1}) = \<Y_3, X_1, Y_2, X_2, Y_1, X_3\>\oplus\<Y_1', X_1'\>.$$ Hence, $\mathcal{V}(D_8) = \<Y_3, Y_2, Y_1\>\oplus\<Y_1'\> = D_4$, and $\mathcal{V}^2(D_8) = \<Y_2\>\oplus\<0\> = D_1$. Next, since $\mathcal{F}^{-1}(D_1) = \<Y_3, Y_2, Y_1, X_3\>\oplus \<Y_1'\> = D_5$ we obtain that $\mathcal{V}(D_5) = \<Y_2, Y_1\>\oplus\<0\> = D_2$ and $\mathcal{V}(D_2) = \<Y_2\>\oplus \<0\> = D_1$. It follows that the final type of $A$ is $\nu_A = \{0, 1, 1, 1\}$ so $$\mu_A = [4, 2, 1].$$

For $A[p] \cong I_{2,1}\oplus (I_{1, 1})^{\oplus 2}$ we can do the similar computations using Lemma \ref{lem:Ig1} and Example \ref{exmp:dim3_eo_anum2} to get $$\mu_A = [4, 3, 2].$$

Therefore, we conclude that the Ekedahl-Oort type $[4, 3, 1]$ corresponds to the indecomposable $\mathrm{BT}_1$ group scheme we denote $I_{4, 3}$.

\label{exmp:dim4_eo_anum3}
\end{exmp}

Combining the observations from the previous examples gives us the following Proposition.

\begin{prop} Let $A$ be a principally polarized abelian fourfold over $k$. Assume that $A$ lies in the $\prank$ zero locus of $\cA_4$, and let $\mu_A$ be the Young type associated with the Ekedahl-Oort type of $A$. Then, there are the following possibilities for the $p$-torsion group scheme $A[p]$ of $A$:
\begin{itemize}
    \item $A[p]\cong I_{4, 1}$, and then $\mu_A = [4]$ and $a(A) = 1$;
    \item $A[p]\cong I_{4, 2}$, and then $\mu_A = [4, 1]$ and $a(A) = 2$;
    \item $A[p]\cong I_{3, 1}\oplus I_{1, 1}$, and then $\mu_A = [4, 2]$ and $a(A) = 2$;
    \item $A[p]\cong I_{2, 1}\oplus I_{2, 1}$, and then $\mu_A = [4, 3]$ and $a(A) = 2$;
    \item $A[p]\cong I_{3, 2}\oplus I_{1, 1}$, and then $\mu_A = [4, 2, 1]$ and $a(A) = 3$;
    \item $A[p]\cong I_{4, 3}$, and then $\mu_A = [4, 3, 1]$ and $a(A) = 3$;
    \item $A[p]\cong I_{2, 1}\oplus (I_{1, 1})^{\oplus 2}$, and then $\mu_A = [4, 3, 2]$ and $a(A) = 3$;
    \item $A[p]\cong (I_{1, 1})^{\oplus 4}$, and then $\mu_A = [4, 3, 2, 1]$ and $a(A) = 4$;
\end{itemize}
where $I_{r, a}$ is the notation for a unique indecomposable symmetric $\mathrm{BT}_1$ group scheme of rank $p^{2r}$ of $\prank$ $0$ and $a$-number $a$. 
\label{prop:eo_classification}
\end{prop}

\begin{proof}
We almost proved this in Lemma \ref{lem:Ig1} and the previous Examples. See \cite{oort_eo_classification}, Section 9 for the classification of symmetric $\mathrm{BT}_1$ group schemes, or \cite{pries_eo_classification}, 4.4 together with \cite{elkinpries}, Remark 5.13.
\end{proof}

The final notice we mention in this section is a result by Ekedahl and van der Geer about the irreducibility of certain Ekedahl-Oort strata.

\begin{prop}
If $\mu$ is an eligible Young type and $\mu \not \in \{[4, 3], [4, 3, 1], [4, 3, 2], [4, 3, 2, 1]\}$, then the Ekedahl-Oort stratum $Z_{\mu}$ is irreducible in $\cA_4$. 
\label{prop:ekedahlvdgeer}
\end{prop}
\begin{proof}
See \cite{ekedahlvdgeer}, Theorem 11.5. 
\end{proof}

\newpage
\section{Ekedahl-Oort type $[4, 3]$}
Smooth genus-$4$ curves can be either hyperelliptic or trigonal. We focus on those over $k$, a field of characteristic two, lying inside the $\tworank$ zero locus. By Proposition \ref{prop:first_eo_prop}, we see that investigating the locus of supersingular curves is related to understanding the Ekedahl-Oort strata defined by a Young type $\mu$ with $\mu \leq [4, 3]$. 

Here we consider the type $[4, 3]$ and describe the curves inside the corresponding locus using explicit computations of the associated Hasse-Witt matrices. By the following result, we find there are no smooth hyperelliptic curves of genus four in characteristic two with the Ekedahl-Oort type $[4, 3]$.

\begin{prop}
Let $C$ be a smooth hyperelliptic curve of genus four defined over a field $k$ of characteristic two. If $\tworank(C) = 0$, then the Ekedahl-Oort type of $C$ is $[4, 2]$. 
\label{prop:eo43_he}
\end{prop}

\begin{proof}
See \cite{vdgeercycle}, Theorem 3.2, for a statement without proof. See also \cite{elkinpries}, Corollary 5.3.
\end{proof}

A non-hyperelliptic smooth curve $C$ of genus four is a trigonal curve. Its model in the canonical embedding in $\P^3$ is an intersection of a quadric and a cubic hypersurface. Over an algebraic closure $\bar{k}$ of $k$ and after a suitable choice of coordinates, we have that any such $C$ (or more precisely, its canonical model) lies either on a \textit{non-singular quadric} $$Q_{ns}: XY + ZT = 0,$$ or on a \textit{quadric cone} (a singular quadric) $$Q_{c}:  XY + T^2 = 0.$$ 
For those lying on a quadric cone, we have the following result. 

\begin{prop}
In characteristic two, no non-hyperelliptic smooth curves of genus four with the Ekedahl-Oort type $[4, 3]$ lie on a quadric cone. 
\label{prop:eo43_cone}
\end{prop}
\begin{proof}
See \cite{ddmt}, Theorem 4.5.
\end{proof}

In the computations proving Proposition \ref{prop:eo43_cone}, it was enough to discuss the Hasse-Witt matrices of curves lying on $Q_c$ using the conditions equivalent to having the type $[4, 3]$. Let us describe the Hasse-Witt matrices of smooth curves of genus four in characteristic two that lie on the non-singular quadric $Q_{ns}$.\\
\\
Let $k = \bar{k}$ be an algebraically closed field of characteristic two.
The \textit{Cartier operator} $\mathcal{C}$ on the space of regular differentials on a genus-$4$ curve $C$ over $k$  is by definition $$\mathcal{C}\left((f_0^2 + f_1^2x) \dd x\right) = f_1\dd x, $$ for a separating variable $x$ of $\kappa(C)$,  $f_0, f_1 \in \kappa(C)$, and $\omega = (f_0^2 + f_1^2x) \dd x \in H^{0}(C, \Omega^1_C)$. 

This operator satisfies the following properties for any function $f \in \kappa(C)$ and regular differentials $\omega, \omega_1, \omega_2 \in H^0(C, \Omega_C^1)$:
 $$\mathcal{C}(\omega_1 + \omega_2) = \mathcal{C}(\omega_1) + \mathcal{C}(\omega_2), \quad \mathcal{C}(f^2\omega) = f\mathcal{C}(\omega), \quad \mathcal{C}(\dd f) = 0,$$ and
    $$\mathcal{C}\left (\frac{\dd f}{f} \right ) = \frac{\dd f}{f}.$$

Given a basis $\{\omega_1, \omega_2, \omega_3, \omega_4\}$ of $H^0(C, \Omega_C^1)$ and the relations $$\mathcal{\Car}(\omega_i) = \sum_{j = 1}^4h_{i j}\cdot \omega_j, \text{ for }i = 1, 2, 3, 4 $$ with $h_{ij} \in k$, the \textit{Hasse-Witt matrix} of $C$ is $$H_C = \left (h_{ij}^2 \right )_{1\leq i, j \leq 4}.$$ Note that, by definition, the rank of $H_C$ equals the rank of the Cartier operator $\mathcal{C}$. Furthermore, using the Cartier operator, we can compute the $\tworank$ and the $a\text{-number}$ of a curve $C$ as $$\tworank(C) = \mathrm{rk}(\mathcal{C}^4) = \dim(\mathrm{Im}(\mathcal{C}^4)),\quad\text{ and}\quad   a(C) = \dim(\ker(\mathcal{C})) = 4 - \mathrm{rk}(\mathcal{C}).$$

In \cite{stohrvoloch}, Section 2, St\"ohr and Voloch gave formulas for computing the Hasse-Witt matrices of non-hyperelliptic smooth genus four curves over $k$. Consider the Segre embedding  
\begin{equation}
\mathbb{P}^1\times \mathbb{P}^1 \overset{\cong}{\to} Q_{ns} = \{(X:Y:Z:T)\in \mathbb{P}^3:  XY = ZT\},
\label{eq:segre_splitquad}    
\end{equation}
and note that $Q_{ns}$ contains the affine plane $\{(a:b:ab:1) \in S: a, b\in \Fbar_2\}$. For any non-hyperelliptic curve lying on $Q_{ns}$, we get an affine equation of the form $$C: f(x, y) = \sum_{i, j = 0}^3a_{i,j}x^iy^j = 0, $$ for which a basis of regular differentials $\{\omega_1, \omega_2, \omega_3, \omega_4\}$ for $C$ can be given as $$\omega_1 = \frac{1}{\partial f/\partial y} \dd x,\quad \omega_2 = \frac{x}{\partial f/\partial y} \dd x,\quad \omega_3 = \frac{y}{\partial f/\partial y} \dd x, \quad\text{and}\quad \omega_4 = \frac{xy}{\partial f/\partial y} \dd x, $$ and the Hasse-Witt matrix expressed as 

\begin{equation}
\label{eq:hassewitt_split_quadric}
H_C = \begin{pmatrix}
a_{11} & a_{31} & a_{13} & a_{33}\\ 
a_{01} & a_{21} & a_{03} & a_{23}\\ 
a_{10} & a_{30} & a_{12} & a_{32}\\ 
a_{00} & a_{20} & a_{02} & a_{22}
\end{pmatrix}.
\end{equation}

Note that the group $G = C_2\ltimes (\PGL_2(k)\times \PGL_2(k))$ acts on $\P^1\times \P^1$, where $C_2$ is the cyclic group of order two which acts by interchanging the $\P^1$'s. Using the isomorphism \eqref{eq:segre_splitquad}, explicitly given by $$((x:y),(z:t))\mapsto (xz: yt:yz:xt),$$ we see that each element of $G$ induces a projective automorphism preserving $Q_{ns}$. See also \cite{savitt_ptsgenF8}, Section 3.

\begin{lem}
Let $k$ be an algebraically closed field of characteristic two. Up to isomorphism, any trigonal, smooth genus-$4$ curve $C$ over $k$ that lies on the non-singular quadric can be written as $V(XY + ZT, q)$ in $\P^3$, with $q\in k[X, Y, Z, T]$ a cubic of the form: 
\begin{enumerate}
    \item[a)] $q = 1\cdot X^3 + 1\cdot Y^3 + a_{21}X^2Y + a_{31}X^2Z + a_{20}X^2T + a_{12}XY^2 + a_{22}XYZ + a_{11}XYT + a_{32}XZ^2 + a_{10}XT^2 + a_{13}Y^2Z + a_{02}Y^2T + a_{23}YZ^2 + a_{01}YT^2$, or
    \item[b)] $q = 1\cdot X^3 + a_{21}X^2Y + a_{31}X^2Z + a_{20}X^2T + a_{12}XY^2 + a_{22}XYZ + a_{11}XYT + a_{32}XZ^2 + a_{10}XT^2 + a_{13}Y^2Z + a_{02}Y^2T + a_{23}YZ^2 + a_{01}YT^2$. 
\end{enumerate}
Moreover, the coefficients $a_{ij} \in k$ appearing in $a)$ and $b)$ are the ones occurring in \eqref{eq:hassewitt_split_quadric} with $a_{00} = a_{33} = 0$, and $a_{30}$ and $a_{03}$ the coefficients of $X^3$ and $Y^3$ respectively. 
\label{lem:split_quad_nice_eqs}
\end{lem}
\begin{proof}
Any $C$ lying on $Q_{ns}\cap V(q)$, for $q$ with a monomial divisible by $ZT$ with a non-zero coefficient, also lies on $Q_{ns}\cap V(q)$ where $q$ does not contain such monomials.

Then, consider the following three types of projective automorphisms preserving $Q_{ns}$ induced by $G$: 
\begin{enumerate}
    \item ${\bigl(\begin{smallmatrix}
X & Y & Z & T\\ 
 X + aZ & Y & Z & T + aY 
\end{smallmatrix}\bigr), \bigl(\begin{smallmatrix}
X & Y & Z & T\\ 
 X & Y + aT & Z + aX & T 
\end{smallmatrix}\bigr), \bigl(\begin{smallmatrix}
X & Y & Z & T\\ 
 X + aT & Y & Z + aY & T 
\end{smallmatrix}\bigr)}$, and ${\bigl(\begin{smallmatrix}
X & Y & Z & T\\ 
 X  & Y + aZ& Z & T + aX 
\end{smallmatrix}\bigr),}$ with $a\in k$;

\item ${\bigl(\begin{smallmatrix}
X & Y & Z & T\\ 
 X & aY & aZ & T  
\end{smallmatrix}\bigr), \bigl(\begin{smallmatrix}
X & Y & Z & T\\ 
 aX & Y & Z & aT 
\end{smallmatrix}\bigr), \bigl(\begin{smallmatrix}
X & Y & Z & T\\ 
 aX& Y & aZ  & T 
\end{smallmatrix}\bigr),}$ and ${\bigl(\begin{smallmatrix}
X & Y & Z & T\\ 
 X  & aY & Z & aT 
\end{smallmatrix}\bigr),}$ with $a\in k$;

\item ${\bigl(\begin{smallmatrix}
X & Y & Z & T\\ 
 X & Y & T & Z 
\end{smallmatrix}\bigr), \bigl(\begin{smallmatrix}
X & Y & Z & T\\ 
 Z & T & X & Y 
\end{smallmatrix}\bigr),}$ and $\bigl(\begin{smallmatrix}
X & Y & Z & T\\ 
 T & Z & X & Y 
\end{smallmatrix}\bigr)$. 
\end{enumerate}
First, assume that at least one of the coefficients with $X^3, Y^3, Z^3$, and $T^3$ is non-zero. After applying a transformation of type $3$, we may assume that the coefficient $a_{30}$ with $X^3$ is non-zero and $a_{30} = 1$ after scaling. Possibly after using the first and then the third type $1$ transformation, we eliminate $Z^3$ and $T^3$ from $q$. If the coefficient with $Y^3$ is zero, $q$ is of type $b)$. Otherwise, applying the first transformation of type $2$, we get a $q$ of type $a)$.

Finally, assume that all the coefficients with $X^3, Y^3, Z^3$, and $T^3$ are zero. Since $q$ is irreducible, at least one of the coefficients with $X^2Z, X^2T$, or $Y^2Z, Y^2T$ is non-zero. After applying some transformation of type $1$, we will end up in an already discussed case. 
\end{proof}

In addition to the preceding, we have the following.

\begin{lem} Let $C = V(XY + ZT, q)$ in $\P^3$ be as in Lemma \ref{lem:split_quad_nice_eqs} with $q \in k[X, Y, Z, T]$ a cubic of form $a)$ or $b)$. The smoothness of $C$ implies $(a_{10}, a_{01}) \neq 0$ and $(a_{23}, a_{32}) \neq 0$.  
\label{lem:easy_singularity}
\end{lem}

\begin{proof}
If $a_{10}= a_{01} = 0$, then $(0:0:0:1)$ is a singular point of $C$, and if $a_{23} = a_{32} =  0$ then $(0:0:1:0)$ is a singular point of $C$.
\end{proof}

Using the previous Lemmas, we are able to show the following Proposition.

\begin{prop}
In characteristic two, no non-hyperelliptic smooth curves of genus four with the Ekedahl-Oort type $[4, 3]$ lie on a non-singular quadric $Q_{ns}$.
\label{prop:eo43_smthquad}
\end{prop}

\begin{proof}
We discuss what the Hasse-Witt matrix $H_C$ of a curve $C$ can be case-by-case and conclude that there are no smooth ones of type $[4, 3]$. Being of type $[4,3]$ is equivalent to the conditions $\rank_k H_C = 2$ (i.e., $\<\mathcal{C}\omega_i: i = 1, 2, 3, 4\>$ is two-dimensional) and $\mathcal{C}^2\omega_i = 0$ for all $i$. By Lemma \ref{lem:split_quad_nice_eqs}, we have two possible forms of $H_C$. 

Let $C = V(XY + ZT, q)$ with $q$ of type $a)$, when we have $$H_C = \begin{pmatrix}
a_{11} & a_{31} & a_{13} & 0\\ 
a_{01} & a_{21} & 1 & a_{23}\\ 
a_{10} & 1 & a_{12} & a_{32}\\ 
0 & a_{20} & a_{02} & a_{22}
\end{pmatrix}.$$ 
\noindent
Assume first $a_{11} = a_{22} = 0$. There are two possibilities.

\begin{itemize}
    \item $\mathcal{C}\omega_2$ and $\mathcal{C}\omega_3$ are linearly independent. Then from $\mathcal{C}^2\omega_1 = \mathcal{C}^2\omega_4 = 0$ we obtain $a_{31} = a_{13} = a_{20} = a_{02} = 0$. Those are in contradiction with $\mathcal{C}^2\omega_2 = 0$.
    \item $\mathcal{C}\omega_2$ and $\mathcal{C}\omega_3$ are linearly dependent. Then $\Car\omega_1$ and $\Car\omega_4$ have to be linearly dependent too. Otherwise, the rank condition would imply $a_{10} = a_{01} = a_{23} = a_{32} = 0$, so $C$ would be singular by Lemma \ref{lem:easy_singularity}. Then, it follows that $a_{21}^3 = 1$. The assumption $a_{01} = 0$ implies $a_{10} = 0$, and $a_{23} = 0$ implies $a_{32} = 0$. In both cases, $C$ is singular by Lemma \ref{lem:easy_singularity}, so we may assume $a_{01}\neq 0$, $a_{23}\neq 0$. However, then $C$ is singular at $(x:\sqrt{a_{21}}x:\sqrt{a_{21}}x^2:1)$, for the choice $x = \left ( {a_{01}}/{a_{23}a_{21}} \right )^{1/4}$.
\end{itemize}

Next, assume $a_{11} \neq 0$ and $a_{22} = 0$; similarly for $a_{11} = 0, a_{22} \neq 0$. If $(a_{20}, a_{02})\neq 0$, then $\Car\omega_2$ and $\Car\omega_3$, and moreover $\Car\omega_1, \Car\omega_2, \Car\omega_3$ are linearly dependent since $\Car^2\omega_1 = 0$, so $a_{23} = a_{32} = 0$ and $C$ is singular by Lemma \ref{lem:easy_singularity}. Therefore $a_{20} = a_{02} = 0$. After writing $\Car\omega_1$ in terms of $\Car\omega_2$ and $\Car\omega_3$ and considering $\Car^2\omega_2 = \Car^2\omega_3 = 0$, we can express all the coefficients $a_{ij}\neq a_{23}, a_{32}$ in terms of $a_{01}, a_{10}$, and $a_{11}$. It is not hard to see that $(\sqrt{a_{10}}:0:0:1)$ will always be a singular point of such a $C$.  

The final possibility is $a_{11} \neq 0$ and $a_{22} \neq 0$. Consider the relations 
\begin{equation}
\Car\omega_1 = \sqrt[4]{\frac{a_{31}}{a_{11}}}\Car\omega_2 + \sqrt[4]{\frac{a_{13}}{a_{11}}}\Car\omega_3, \quad \text{and} \quad \Car\omega_4 = \sqrt[4]{\frac{a_{20}}{a_{22}}}\Car\omega_2 + \sqrt[4]{\frac{a_{02}}{a_{22}}}\Car\omega_3.
\label{eqn:Car_relations1}    
\end{equation}
From them and the rank condition, it follows that $\Car\omega_2$ and $\Car\omega_3$ are linearly independent, and we get eight relations 

\begin{equation}
\left\{\begin{matrix}
a_{11} =& \sqrt{\alpha}\cdot a_{01} + \sqrt{\beta}\cdot a_{10}\\ 
\alpha\cdot a_{11} =& \sqrt{\alpha}\cdot a_{21} + \sqrt{\beta}\\ 
\beta\cdot a_{11} =& \sqrt{\alpha} + \sqrt{\beta}\cdot a_{12}\\ 
0 =& \sqrt{\alpha}\cdot a_{23} + \sqrt{\beta}\cdot a_{32}\\ 
\end{matrix}\right., \quad \left\{\begin{matrix}
0 =& \sqrt{\gamma}\cdot a_{01} + \sqrt{\delta}\cdot a_{10}\\ 
\gamma\cdot a_{22} =& \sqrt{\gamma}\cdot a_{21} + \sqrt{\delta}\\ 
\delta\cdot a_{22} =& \sqrt{\gamma} + \sqrt{\delta}\cdot a_{12}\\ 
a_{22} =& \sqrt{\gamma}\cdot a_{23} + \sqrt{\delta}\cdot a_{32}\\ 
\end{matrix}\right.;
\label{eqn:Car_relations2}    
\end{equation}
where $$a_{31} = \alpha a_{11},\quad  a_{13} = \beta a_{11},\quad a_{20} = \gamma a_{22}, \text{ }\text{ and } \text{ } a_{02} = \delta a_{22},$$ for some $\alpha, \beta, \gamma, \delta 
\in k$. Note that $\alpha \neq 0, \beta \neq 0, \gamma \neq 0$, and $\delta \neq 0$ because of \eqref{eqn:Car_relations1} and $\Car^2\omega_1 = \Car^2\omega_4 = 0$, while $\Car\omega_1 \neq 0$ and $\Car\omega_4\neq 0$. 
If we write down $\Car^2\omega_2 = \Car^2\omega_3 = 0$ in terms of generators $\Car\omega_2$ and $\Car\omega_3$, we get another four relations $$\left\{\begin{matrix}
a_{21} =& {\alpha}\cdot a_{01} + {\gamma}\cdot a_{23}\\ 
1 =& {\alpha}\cdot a_{10} + {\gamma}\cdot a_{32}\\ 
1 =& {\beta}\cdot a_{01} + {\delta}\cdot a_{23}\\ 
a_{12} =& {\beta}\cdot a_{10} + {\delta}\cdot a_{32}\\ 
\end{matrix}\right..$$
Denote $\Delta = \det 
\begin{pmatrix}
\sqrt{\alpha} & \sqrt{\beta}\\ 
\sqrt{\gamma} & \sqrt{\delta} 
\end{pmatrix}$, and assume $\Delta \neq 0$. Solving the equations leads to $$a_{01} = \frac{a_{11}\sqrt{\delta}}{\Delta},\quad  a_{10} = \frac{a_{11}\sqrt{\gamma}}{\Delta},\quad  a_{23} = \frac{a_{22}\sqrt{\beta}}{\Delta} ,\quad a_{32} = \frac{a_{22}\sqrt{\alpha}}{\Delta},$$ $$a_{21} = \frac{\alpha\sqrt{\delta}a_{11} + \gamma\sqrt{\beta}a_{22}}{\Delta},\quad a_{12} = \frac{\sqrt{\alpha}{\delta}a_{22} + \sqrt{\gamma}{\beta}a_{11}}{\Delta}, \text{ }\text{ and}$$ $$ \frac{\alpha\sqrt{\gamma} a_{11} + \sqrt{\alpha}\gamma a_{22}}{\Delta} = 1 = \frac{\beta\sqrt{\delta} a_{11} + \sqrt{\beta}\delta a_{22}}{\Delta}.$$
It is not hard to see that the point $P = (x: y: xy: 1)$, with $(x, y) \in k^2$ a solution to the equation $\left\{\begin{matrix}
1 + \alpha x^2 + \beta y^2 = 0\\ 
\gamma x^2 + \delta y^2 + x^2y^2 = 0
\end{matrix}\right.$, is a singular point of such a $C$. Otherwise, $\Delta = 0$, so there is some $\lambda \in k^*$, so that $\gamma = \lambda \alpha, \delta = \lambda \beta$. 
The top relations in the left and right columns of \eqref{eqn:Car_relations2} imply that this cannot happen.

Let us now discuss whether a smooth genus-$4$ curve $C = V(XY + ZT, q)$ with $q$ of type $b)$ can be of the type $[4, 3]$. The Hasse-Witt matrix of $C$ is now $$H_C = \begin{pmatrix}
a_{11} & a_{31} & a_{13} & 0\\ 
a_{01} & a_{21} & 0 & a_{23}\\ 
a_{10} & 1 & a_{12} & a_{32}\\ 
0 & a_{20} & a_{02} & a_{22}
\end{pmatrix}.$$ 
With similar arguments as above, the cases $a_{11} =  0$ or $a_{22} = 0$ cannot occur for such a smooth curve $C$. As previously, we assume $a_{11} \neq 0, a_{22} \neq 0$, and write $$a_{31} = \alpha a_{11}, \quad a_{13} = \beta a_{11},\quad a_{20} = \gamma a_{22}, \text{ }\text{ and }\text{ } a_{02} = \delta a_{22}, $$ for some $\alpha, \beta, \gamma, \delta \in k$. Since $\Car^2\omega_1 = \Car^2\omega_4 = 0$, $\Car\omega_1 \neq 0$ and $\Car\omega_4 \neq 0$ it follows that $\alpha\neq 0$ and $\gamma \neq 0$. Moreover, we can assume $\Delta \neq 0$ with $\Delta = \det 
\begin{pmatrix}
\sqrt{\alpha} & \sqrt{\beta}\\ 
\sqrt{\gamma} & \sqrt{\delta} 
\end{pmatrix}$ since $a_{11}$ would be zero otherwise (by discussing the relations similar to the ones in \eqref{eqn:Car_relations2}), as well as $\beta\neq 0$ and $\delta \neq 0$ since it is not hard to see that $C$ would be singular otherwise (by looking at the charts $\{Z \neq 0\}$ and $\{T \neq 0\}$).
From the relations \eqref{eqn:Car_relations1}, and $\Car^2\omega_2 = \Car^2\omega_3 = 0$ we find that $$ a_{01}= \frac{a_{11}\sqrt{\delta}}{\Delta},\quad a_{10}= \frac{a_{11}\sqrt{\gamma}}{\Delta},\quad a_{23}= \frac{a_{22}\sqrt{\beta}}{\Delta},\quad  a_{32}= \frac{a_{22}\sqrt{\alpha}}{\Delta}, $$
$$a_{12} = \sqrt{\beta}\cdot a_{11} = \sqrt{\delta}\cdot a_{22}, 
\quad a_{21} = \frac{a_{11}\Delta}{\sqrt{\delta}} = \frac{a_{22}\Delta}{\sqrt{\beta}}, \quad a_{11} = \sqrt{\frac{\delta}{\alpha\gamma}},  \text{ }\text{ and } \text{ } a_{22} = \sqrt{\frac{\beta}{\alpha\gamma}}.$$

Similarly as in the first part of this proof and using (all of) the obtained formulas, we find that the point $P = (x: y: xy: 1)$, with $(x, y) \in k^2$ a solution to the equation $$\left\{\begin{matrix}
1 + \alpha x^2 + \beta y^2 = 0\\ 
\gamma x^2 + \delta y^2 + x^2y^2 = 0
\end{matrix}\right.$$ is a singular point of such $C$. That concludes our case-by-case analysis. 
\end{proof}

\begin{rem}
Let $\alpha, \beta, \gamma, \delta, \Delta, a_{11}$, and $a_{22}$ be non-zero elements of $k$ and $q = q(X, Y, Z, T)$ a cubic polynomial as above. Denote $f(x, y) = q(x, y, xy, 1)$. In the proof of both parts of the previous Proposition, we used that a point $P = (p_x:p_y:p_xp_y:1)\in \P^3$ is a singular point of $C$ if and only if $$\frac{\partial f}{\partial x}(p_x, p_y) = 0, \quad \frac{\partial f}{\partial y}(p_x, p_y) = 0, \text{ }\text{ and }\text{ }(f + x\frac{\partial f}{\partial x}+ y\frac{\partial f}{\partial y})(p_x, p_y) = 0.$$ Note that  $$f + x\frac{\partial f}{\partial x}+ y\frac{\partial f}{\partial y} = a_{11}(1 + \alpha x^2 + \beta y^2) + a_{22}(\gamma x^2 + \delta y^2 + x^2y^2).$$ 
\end{rem}

As a consequence of the results presented in this section, we have the following Corollary.

\begin{cor}
In characteristic two, there are no smooth curves of genus four with the Ekedahl-Oort type $[4, 3]$.
\label{cor:no_43_eo_smth_curves}
\end{cor}
\begin{proof}
This follows from Proposition \ref{prop:eo43_he}, Proposition \ref{prop:eo43_cone}, and Proposition \ref{prop:eo43_smthquad}.
\end{proof}

We end this section with two remarks on the result obtained in Corollary \ref{cor:no_43_eo_smth_curves}.

\begin{rem}
In contrast to Corollary \ref{cor:no_43_eo_smth_curves}, in \cite{zhou_genus4}, Theorem 1.2,  Zhou shows the existence of a smooth genus-$4$ curve over $\Fbar_p$ with the Ekedahl-Oort type $[4, 3]$ for any odd prime number $p$ with $p \equiv \pm 2 \mod 5$. 
\end{rem}

\begin{rem}
\label{rem:eo43_implies_ss3dim}
Note that $Z_{[4, 3]}\cap j(\partial {\cM_g^{ct}})$ is a $2$-dimensional locus - the only singular, stable genus-$4$ curves with the type $ [4, 3]$ are the ones whose two components are two supersingular genus-$2$ curves; this follows from Proposition \ref{prop:eo_classification}. We can combine that observation with Corollary \ref{cor:no_43_eo_smth_curves} and Example \ref{ex:smooth_ss_curve} to give the second proof of Theorem \ref{thm:supersingular_curves_dim3}, using that $Z_{[4, 3]}\subseteq \cS_4$ is a $3$-dimensional locus.
\end{rem}
\newpage
\section{Generic $a$-number}

Let $k$ be an algebraically closed field in characteristic $p>0$, and consider $\cM_g = \cM_g\otimes \Fbar_p$. In \cite{pries_a_number}, using the induction and the intersection with the boundary strata, Pries shows that the generic point of any component of the $\prank \leq f$ locus of smooth genus-$g$ curves $$V_{f} \cM_g,$$ for $f = g- 2$ or $f = g- 3$, has $a$-number one. Having Pries's inductive argument, which we present below, it was enough to show that the generic points of all components of $V_0 \cM_2$ and $V_0 \cM_3$ have $a$-number one.

\begin{prop}
Let $g\geq 2$, $1\leq f <g$, and $\cM_g = \cM_g \otimes \Fbar_p$. If the generic point $X$ of any component of $V_f \cM_g$ has $a(X) = 1$, then the generic point $Y$ of any component of $V_{f + 1} \cM_{g + 1}$ has $a(Y) = 1$. 

In particular, if the generic point $X$ of any component of the locus $V_0 \cM_g$ has $a(X) = 1$, then for any $h\geq g$, the generic point $Y$ of any component of $V_{h-g} \cM_h$ has $a(Y) = 1$.
\label{prop:pries_induction}
\end{prop}

\begin{proof}
See \cite{pries_a_number}, Proposition 3.7. 
\end{proof}

Here, we show that for $p = 2$, the generic point of any component of $V_0\cap \cJ_4$ has $a$-number one. We start by showing that for $p = 2$, there is a smooth genus-$4$  curve $C$ with $\prank(C) = 0$ and $a(C) = 1$.

\begin{exmp}
The curve $C:\left\{\begin{matrix}
XY + ZT = 0\\ 
X^2Z + Y^2Z + YZ^2 + X^2T + Y^2T + XT^2 = 0
\end{matrix}\right. \text{ in } \P^3 $ from Example \ref{ex:smooth_ss_curve}   has $a$-number $a(C) = 1$. 

That follows from \eqref{eq:hassewitt_split_quadric}, since $$H_C = \begin{pmatrix}
0 & 1 & 1 & 0\\ 
0 & 0 &0 & 1\\ 
1 & 0 & 0 & 0\\ 
0 &1 & 1 & 0
\end{pmatrix},$$ and thus $a(C) = 4 - \rank H_C = 1$.
\label{exmp:ss_prank0_anum1}
\end{exmp}

\noindent A specificity of working with curves $C$ in positive characteristics is that there are certain non-trivial constraints among $p, g = g(C)$, and $a = a(C)$ that should be satisfied for such curves $C$ to exist. For example, we have the following result by Zhou. 

\begin{prop}
Let $C$ be a smooth genus-$g$ curve defined over a field of characteristic $p$. If the rank of the Cartier operator of $C$ equals $1$, i.e., if $a(C) = g - 1$, then $$g \leq p + \frac{p(p-1)}{2}.$$
\label{prop:zhou_a}
\end{prop}

\begin{proof}
See \cite{zhou_a_number}, Theorem 1.1.
\end{proof}

Using the mentioned results in case $p = 2$, we extract information about the $a$-numbers of some $2$-rank strata.

\begin{thm} Let $\cJ_4 = \cJ_4\otimes \Fbar_2$  and let $V_0\subseteq \cA_4=\cA_4 \otimes \Fbar_2$ be the $\tworank$ zero locus. Let $X$ be a generic point of a component of $V_0\cap \cJ_4$, the $\tworank$ zero locus of Jacobians. Then $a(X) = 1$.
\label{thm:generic_a_num}
\end{thm}

\begin{proof}
Since there is a smooth genus-$4$ curve over $\Fbar_2$ with $\prank$ $0$ and $a$-number $1$, by Example \ref{exmp:ss_prank0_anum1}, we see that $Z_{[4]}\cap \cJ_4\neq 0$, and thus that the Ekedahl-Oort stratum $Z_{[4]}\cap \cJ_4$ needs to be of dimension $$\dim(Z_{[4]}\cap \cJ_4) = \dim(V_0\cap \cJ_4) = 5.$$ Therefore, to get the result, we need to show that all components of $\overline{Z_{[4, 1]}}\cap \cJ_4$ have dimension $\leq 4$. 

For a contradiction, assume that $\overline{Z_{[4, 1]}}\subseteq \cJ_4$ is of dimension $5$, using Proposition \ref{prop:ekedahlvdgeer}. Then all the Ekedahl-Oort strata $Z_{\mu}$ for $\mu \leq [4, 1]$ are contained in $\cJ_{4}$. In particular, $Z_{[4, 3, 1]}$ is a two-dimensional non-empty subvariety of $\cJ_4$. By Proposition \ref{prop:eo_classification} and using that $I_{4, 3}$ is an indecomposable symmetric $\mathrm{BT}_1$ group scheme, it follows that $Z_{[4, 3, 1]}$ is contained in the open Torelli locus $\cJ_4^0$. However, by Proposition \ref{prop:zhou_a}, there are no smooth genus-$4$ curves over $k$ with $a$-number $3$. Hence $\dim \overline{Z_{[4, 1]}}\cap \cJ_4 \leq 4$. 
\end{proof}

\begin{rem}
\label{rem:supersingular_locus_3dim}
The way we used the type $[4, 3, 1]$ in the previous argument can also be employed, together with Example \ref{ex:smooth_ss_curve}, to give the third proof of Theorem \ref{thm:supersingular_curves_dim3}.
\end{rem}

As a consequence of the previous argument, we obtain the following result.

\begin{cor} Let $\cJ_4 = \cJ_4 \otimes \Fbar_2$. Then the Ekedahl-Oort strata $Z_{[4]}\cap \cJ_4$, $Z_{[4, 1]}\cap \cJ_4$, and $Z_{[4, 2]}\cap \cJ_4$ are respectively of the expected codimensions $4, 5,$ and $6$ in $\cJ_4$, while $Z_{\mu}\cap \cJ_4^0 = \o$ exactly for $\mu \in \{[4, 3], [4, 2, 1], [4, 3, 1], [4, 3, 2], [4, 3, 2, 1]\}$.  
\label{cor:eo_conclusion}
\end{cor}
\begin{proof}
Note that $Z_{[4, 2]}\cap \cJ_4 \neq \o$ since the hyperelliptic curves of $\tworank = 0$ have the type $[4, 2]$, and the curve $$C: y^2 + y = x^9 + x^5$$ considered in Example \ref{ex:isog_class_ss} is a supersingular curve, so a curve of $\tworank$ zero.  The result for the type $[4, 2]$ follows by the argument analogous to the one above. A similar argument can also be used for the type $[4, 1]$ since, by Example \ref{exmp:EO41nonempty}, we see that the curve $C$ given by the equations $$C:\left\{\begin{matrix}
XY + T^2 = 0\\ 
TX^2 + Y^3 + X^2Z + Z^3 = 0
\end{matrix}\right. \text{ in } \P^3,$$ has the type $[4, 1]$. The conclusion for $\mu = [4]$ was already made in the proof of Theorem \ref{thm:generic_a_num}. Lastly, the result for $\mu \leq [4, 2, 1]$ follows from Proposition \ref{prop:zhou_a}, and for $\mu = [4, 3]$, it follows from Corollary \ref{cor:no_43_eo_smth_curves}.
\end{proof}

The argument offered in the proof of Theorem \ref{thm:generic_a_num} gives us the following observation. 

\begin{rem}
Let $p>0$ be an arbitrary prime number and let $[4, 3, 1]$ be the Ekedahl-Oort type in characteristic $p$ and dimension $4$, and $\cJ_4^0 = \cJ_4^0 \otimes \Fbar_p$. Assume that $Z_{[4, 3, 1]} \not \subseteq \cJ_4^0$ (e.g., that the locus of smooth genus-$4$ curves with the Ekedahl-Oort type $[4, 3, 1]$ in characteristic $p$ is not $2$-dimensional). Then the analog of Theorem \ref{thm:generic_a_num} holds: for any generic point $X$ of a component $(V_0\cap \cJ_4)\otimes \Fbar_p$, we have $a(X) = 1$. Moreover, that implies the existence of a smooth genus-$4$ curve in characteristic $p$ with the Ekedahl-Oort type $[4]$.
\end{rem}

\begin{rem}
By Theorem \ref{thm:supersingular_curves_dim3}, Corollary \ref{cor:np_conclusion}, and Corollary \ref{cor:eo_conclusion}, we know the dimensions (of any component) of the intersection of $\cJ_4$ and Newton and Ekedahl-Oort strata inside $V_0$ in characteristic two. However, we cannot say much about their irreducibility or the number of irreducible components. In fact, we do not even know that for $V_0 \cap \cJ_4$. For now, we only know that $V_0 \cap \cJ_g$ is connected for any $g\geq 2$ in any characteristic $p>0$ by Achter and Pries in \cite{achterpries_prankconn}, Corollary 4.5 using \cite{achterpries_monodromy}, Proposition 3.4.
\end{rem}

\end{document}